\documentclass{amsart}
\usepackage{amssymb, amsmath, amsthm, graphics, comment, xspace, enumerate}
\usepackage{hyperref}
\newtheorem{thm}{Theorem}[section]

\newtheorem{lem}[thm]{Lemma}
\newtheorem{prop}[thm]{Proposition}
\theoremstyle{definition}
\newtheorem{defn}[thm]{Definition}
\newtheorem{example}[thm]{Example}
\theoremstyle{remark}

\numberwithin{equation}{section}

\begin{document}
\title[Remotely $c$-almost periodic type functions in ${\mathbb R}^{n}$]{Remotely $c$-almost periodic type functions in ${\mathbb R}^{n}$}

\author{M. Kosti\' c}
\address{Faculty of Technical Sciences,
University of Novi Sad,
Trg D. Obradovi\' ca 6, 21125 Novi Sad, Serbia}
\email{marco.s@verat.net}

\author{V. Kumar}
\address{Max-Planck Institute for Dynamics of Complex Technical Systems, Sandtorstr. 1, 39106 Magdeburg, Germany}
\email{math.vipinkumar219@gmail.com}

{\renewcommand{\thefootnote}{} \footnote{
\\ \text{  }  \ \    {\it Key words and phrases.} Remotely $c$-almost periodic functions in ${\mathbb R}^{n},$ slowly oscillating functions in ${\mathbb R}^{n},$
quasi-asymptotically $c$-almost periodic functions in ${\mathbb R}^{n},$
abstract Volterra integro-differential equations, Richard-Chapman ordinary differential equation with external perturbation.
\\  \text{  }  
Marko Kosti\' c is partially supported by grant 451-03-68/2020/14/200156 of Ministry
of Science and Technological Development, Republic of Serbia.}}

\begin{abstract}
In this paper, we relate the notions of remote
almost periodicity and quasi-asymptotical almost periodicity; in actual fact, we observe that a remotely almost periodic function is nothing else but a bounded, uniformly continuous quasi-asymptotically almost periodic function.
We introduce and analyze several new classes of remotely $c$-almost periodic functions in ${\mathbb R}^{n},$ slowly oscillating functions in ${\mathbb R}^{n},$
and
further analyze the recently introduced class of quasi-asymptotically $c$-almost periodic functions in ${\mathbb R}^{n}.$
We provide certain applications of our theoretical results to
the abstract Volterra integro-differential equations and the ordinary differential equations.
\end{abstract}
\maketitle

\section{Introduction and preliminaries}

The notion of almost periodicity was introduced by the Danish mathematician H. Bohr around 1924--1926 and later reconsidered by many others. 
Let $(X,\| \cdot \|)$ be a complex Banach space, and let $F : {\mathbb R}^{n} \rightarrow X$ be a continuous function ($n\in {\mathbb N}$). Then it is said that $F(\cdot)$ is almost periodic if and only if for each $\epsilon>0$
there exists $l>0$ such that for each ${\bf t}_{0} \in {\mathbb R}^{n}$ there exists ${\bf \tau} \in B({\bf t}_{0},l)$ with
\begin{align*}
\bigl\|F({\bf t}+{\bf \tau})-F({\bf t})\bigr\| \leq \epsilon,\quad {\bf t}\in {\mathbb R}^{n},
\end{align*}
where $B({\bf t}_{0},l)$ denotes the closed ball in ${\mathbb R}^{n}$ with center ${\bf t}_{0}$ and radius $l>0.$
Equivalently, $F(\cdot)$ is almost periodic if and only if  for any sequence $({\bf b}_n)$ in ${\mathbb R}^{n}$ there exists a subsequence $({\bf a}_{n})$ of $({\bf b}_n)$
such that $(F(\cdot+{\bf a}_{n}))$ converges in $C_{b}({\mathbb R}^{n}: X),$ the Banach space of bounded continuous functions $F :{\mathbb R}^{n}\rightarrow X$ equipped with the sup-norm. Any trigonometric polynomial in ${\mathbb R}^{n}$ is almost periodic and it is also well known that $F(\cdot)$ is almost periodic if and only if there exists a sequence of trigonometric polynomials in ${\mathbb R}^{n}$ which converges uniformly to $F(\cdot).$

Any almost periodic function $F(\cdot)$ is bounded and uniformly continuous. Further on,
if $F : {\mathbb R}^{n} \rightarrow X$ is an almost periodic function, then $F(\cdot)$ is uniformly recurrent in the sense that $F(\cdot)$ is continuous and there exists a sequence $({\bf \tau}_{k})$ in ${\mathbb R}^{n}$ such that $\lim_{k\rightarrow +\infty} |{\bf \tau}_{k}|=+\infty$ and
$$
\lim_{k\rightarrow +\infty}\sup_{{\bf t}\in {\mathbb R}^{n}} \bigl\|F\bigl({\bf t}+{\bf \tau}_{k}\bigr)-F({\bf t})\bigr\| =0.
$$
In \cite{marko-manuel-ap}, a joint paper with A. Ch\'avez, K. Khalil and M. Pinto, the first named author has analyzed
various classes of almost periodic functions of form $F : I \times X\rightarrow Y,$ where $(Y,\|\cdot \|_{Y})$ is a complex Banach spaces and $\emptyset \neq  I \subseteq {\mathbb R}^{n}$.
For more details about almost periodic functions and their applications, we refer the reader to \cite{besik, marko-manuel-ap, diagana, fink, gaston, nova-mono, nova-selected, 188, pankov, 30}.

On the other hand, the class of \emph{S}-asymptotically $\omega$-periodic functions was thoroughly analyzed by H. R. Henr\'iquez, M. Pierri and P. T\' aboas \cite{pierro} in 2008 (for some applications of 
$S$-asymptotically $\omega$-periodic functions, we refer the reader to \cite{prc-marek, cuevas-souza, dimbour, guengai, pierro,xie} and references quoted therein). 
In \cite{brazil}, we have recently introduced and analyzed the class of quasi-asymptotically almost periodic functions following the approach of A. S. Kovanko \cite{kovanko}. Any
$S$-asymptotically $\omega$-periodic function 
$f: I \rightarrow X$ is quasi-asymptotically almost periodic, while the converse statement is not true in general (\cite{brazil}). 

Further on, in our joint research studies \cite{c1}-\cite{BIMV} with M. T. Khalladi, A. Rahmani, M. Pinto and D. Velinov,
the classes of (Stepanov, Weyl)
$c$-almost periodic type functions, quasi-asymptotically $c$-almost periodic type functions and $S$-asymptotically $(\omega ,c)$-periodic type functions have been examined in the one-dimensional setting. In the research articles \cite{multi-ce}-\cite{multi-arx}, the authors have analyzed multi-dimensional $c$-almost periodic type functions and various classes of multi-dimensional $(\omega,c)$-almost periodic type functions. 
Here we observe that a bounded continuous function $F: {\mathbb R}^{n} \rightarrow X$ is 
remotely almost periodic if and only if $F(\cdot)$ is uniformly continuous and quasi-asymptotically almost periodic.
We investigate various notions of remote $c$-almost periodicity in ${\mathbb R}^{n}$ following the approach obeyed in our recent research article \cite{multi-arx},
where we have analyzed quasi-asymptotically $c$-almost periodic type functions in ${\mathbb R}^{n}.$ The functions under our consideration are defined on a general region $I\subseteq {\mathbb R}^{n}$,  which need not contain the zero vector or be closed under the pointwise addition.  

It is said that a bounded continuous function $F: {\mathbb R}^{n} \rightarrow X$ is slowly oscillating if and only if for each $\omega \in {\mathbb R}^{n}$ we have $\lim_{|{\bf t}| \rightarrow +\infty}\| F({\bf t}+\omega)-F({\bf t})\|=0.$
We know that any slowly oscillating function $F(\cdot)$ is uniformly continuous. Concerning applications of slowly oscillating functions made so far, let us recall that 
F. Yang and C. Zhang have analyzed  
slowly oscillating solutions of parabolic inverse problems in \cite{f-yang1}; see also the research articles \cite{c-zhang-guo, c-zhang1, c-zhang}, and \cite{s-zhang} by S. Zhang, D. Piao.
In this paper, we introduce and analyze some new classes of 
slowly oscillating functions in ${\mathbb R}^{n},$
and further analyze multi-dimensional quasi-asymptotically $c$-almost periodic functions; for simplicity, we will not consider the corresponding Stepanov classes here.

The organization and main ideas of this paper can be briefly described as follows. In Section \ref{so-ce}, we introduce and analyze 
slowly oscillating type functions in ${\mathbb R}^{n}.$ The notion of a $({\mathbb D},{\mathcal B})$-slowly oscillating function is introduced in Definition \ref{drasko-norm}; the main aim of Proposition \ref{koq} is to indicate that the class of $({\mathbb D},{\mathcal B},c)$-slowly oscillating functions
is not interesting for further investigations, if $c\in {\mathbb C} \setminus \{0\}$ and $c\neq 1.$ In Proposition \ref{whiskey}, we extend the well known result of D. Sarason \cite[Proposition 1]{sarason1} concerning the uniform continuity of slowly oscillating functions to the multi-dimensional setting.
The main aim of Example \ref{pier} is to show the existence of a vector-valued slowly oscillating function with not relatively compact range. The class of $({\mathcal B},({\mathbb D}_{j})_{j\in {\mathbb N}_{n}})$-slowly oscillating functions is introduced in
Definition \ref{drasko-arx}; after that, we explain how many structural results established in \cite{multi-arx} can be used to provide certain characterizations of  $({\mathbb D},{\mathcal B})$-slowly oscillating functions and $({\mathcal B},({\mathbb D}_{j})_{j\in {\mathbb N}_{n}})$-slowly oscillating functions.
Concerning the usual class of one-dimensional slowly oscillating functions, we would like to note that 
the statement of \cite[Lemma 2.1]{brindle}, which has recently been established by D. Brindle in his doctoral dissertation and which plausibly holds for uniformly integrable resolvent operator families under consideration, and the statement of \cite[Theorem 3.9]{y-k chang} with $k=0,$ which has recently been proved by 
Y.-K. Chang and Y. Wei,
can be used to profile important results concerning the invariance of slowly oscillating property under the actions of finite convolution products and the actions of infinite convolution products, 
respectively (the multi-dimensional analogues can be deduced without any substantial difficulties; see \cite{marko-manuel-ap} for more details). Such results enable one to analyze the existence and uniqueness of slowly oscillating solutions for various classes of the abstract Volterra integro-differential equations considered in \cite{nova-mono} and \cite{nova-selected}; see also Example \ref{mp} below.

Section \ref{qaap} essentially continues our recent 
analysis of multi-dimensional quasi-asymptotically $c$-almost periodic functions (\cite{multi-arx}).
In this section, we state the fundamental relations between the notions quasi-asymptotical $c$-almost periodicity and remote $c$-almost periodicity; we introduce various notions of 
${\mathbb D}$-remotely $({\mathcal B},I',c)$-almost periodicity in Definition \ref{ris}. Our main aim is, in fact, to show how the already known results established for remotely almost periodic functions (see e.g., \cite{c-zhang1,s-zhang}) can be used in further analysis of quasi-asymptotically almost periodic functions  and how the already established results for quasi-asymptotically ($c$-)almost periodic functions (\cite{BIMV, brazil, multi-arx}) can be used for giving new characterizations of
remotely  ($c$-)almost periodic functions, and especially, for giving new important applications of this class of functions to the abstract Volterra integro-differential equations (see also Proposition \ref{zivotlajf}, Proposition \ref{gasdas1}, Example \ref{merak} and Example \ref{mp}). Therefore, this paper can be viewed as a certain addendum to the paper \cite{multi-arx}.
Further on, we provide several arguments showing that the result of \cite[Proposition 2.4]{s-zhang} is not correct and provide an example of a vector-valued slowly oscillating function without mean value. In the final section of paper, we provide certain
applications to the integro-differential equations; especially, in Subsection \ref{biology}, we continue our recent analysis from \cite{maulen}, a joint work with C. Maul\'en, S. Castillo and M. Pinto, by analyzing slowly oscillating solutions for the Richard-Chapman ordinary differential equation with an external perturbation, which plays an important role in mathematical biology. 

We use the standard notation throughout the paper. By $(X,\| \cdot \|)$ and
$(Y,\| \cdot \|_{Y})$ we denote two complex Banach spaces;
$L(X,Y)$ stands for the Banach algebra of all bounded linear operators from $X$ into
$Y$ with $L(X,X)$ being abbreviated to $L(X)$. 
By $\langle \cdot, \cdot \rangle$ we denote the usual inner product in ${\mathbb R}^{n}.$ If ${\bf t_{0}}\in {\mathbb R}^{n}$ and $\epsilon>0$, then we define $B({\bf t}_{0},\epsilon):=\{{\bf t } \in {\mathbb R}^{n} : |{\bf t}-{\bf t_{0}}| \leq \epsilon\},$
where $|\cdot|$ denotes the Euclidean norm in ${\mathbb R}^{n};$ $(e_{1},e_{2},...,e_{n})$ denotes the standard basis of ${\mathbb R}^{n}.$ Set 
${\mathbb N}_{n}:=\{1,..., n\}.$

We will always assume henceforth that ${\mathcal B}$ is a collection of non-empty subsets of $X$ such that, for  every $x\in X,$ there exists $B\in{\mathcal B}$ with $x\in B.$ 
Suppose that 
${\mathbb D} \subseteq I \subseteq {\mathbb R}^{n}$ and the set ${\mathbb D}$  is unbounded. By $C_{0,{\mathbb D},{\mathcal B}}(I \times X :Y)$ we denote the vector space consisting of all continuous functions $Q : I \times X \rightarrow Y$ such that, for every $B\in {\mathcal B},$ we have $\lim_{t\in {\mathbb D},|t|\rightarrow +\infty}Q({\bf t};x)=0,$ uniformly for $x\in B$ (\cite{marko-manuel-ap}).

The class of $(\omega,c)$-periodic functions was introduced by E. Alvarez, A. G\'omez, M. Pinto in \cite{alvarez1} and later reconsidered by  E. Alvarez, S. Castillo, M. Pinto in \cite{alvarez2} and E. Alvarez, S. Castillo, M. Pinto in \cite{alvarez3}, among many other research articles.
Before proceeding further, we need to
recall the definitions of an
$(S,{\mathbb D},{\mathcal B})$-asymptotically 
$(\omega,c)$-periodic function and an $(S,{\mathcal B})$-asymptotically $(\omega_{j},c_{j},{\mathbb D}_{j})_{j\in {\mathbb N}_{n}}$-periodic function
(cf. also M. T. Khalladi, M. Kosti\' c, M. Pinto, A. Rahmani, D. Velinov \cite[Definition 3.1]{BIMV} and Y.-K. Chang, Y. Wei \cite[Definition 3.1, Definition 3.2]{y-k chang} for the one-dimensional case, as well as M. Kosti\' c \cite{nds-2021} for the multi-dimensional case): 

\begin{defn}\label{drasko-presing-12345} (\cite{multi-arx})
Let ${\bf \omega}\in {\mathbb R}^{n} \setminus \{0\},$ $c\in {\mathbb C} \setminus \{0\},$
${\bf \omega}+I \subseteq I,$ ${\mathbb D} \subseteq I \subseteq {\mathbb R}^{n}$ and the set ${\mathbb D}$ be unbounded. A continuous
function $F:I \times X\rightarrow Y$ is said to be $(S,{\mathbb D},{\mathcal B})$-asymptotically 
$(\omega,c)$-periodic if and only if for each $B\in {\mathcal B}$ we have
\begin{align*}
\lim_{|{\bf t}|\rightarrow +\infty,{\bf t}\in {\mathbb D}}\bigl\| F({\bf t}+\omega ; x)-cF({\bf t} ; x)\bigr\|_{Y}=0,\quad \mbox{ uniformly in }x\in B.
\end{align*}
\end{defn}

\begin{defn}\label{drasko-presing123456} (\cite{multi-arx})
Let ${\bf \omega}_{j}\in {\mathbb R} \setminus \{0\},$ $c_{j}\in {\mathbb C} \setminus \{0\},$
${\bf \omega}_{j}e_{j}+I \subseteq I$,
${\mathbb D}_{j} \subseteq I \subseteq {\mathbb R}^{n}$ and the set ${\mathbb D}_{j}$ be unbounded ($1\leq j\leq n$).
A continuous
function $F:I \times X\rightarrow Y$ is said to be  $(S,{\mathcal B})$-asymptotically $({\bf \omega}_{j},c_{j},{\mathbb D}_{j})_{j\in {\mathbb N}_{n}}$-periodic if and only if for each $ j\in {\mathbb N}_{n}$ we have
\begin{align*}
\lim_{|{\bf t}|\rightarrow +\infty,{\bf t}\in {\mathbb D}_{j}}\bigl\|F({\bf t}+{\bf \omega}_{j}e_{j};x)-c_{j}F({\bf t};x)\bigr\|_{Y}=0,\quad \mbox{ uniformly in }x\in B.
\end{align*}
\end{defn} 

\section{Slowly oscillating type functions in ${\mathbb R}^{n}$}\label{so-ce}

We start this section by introducing the following notion (see also \cite[Definition 4.2.1, p. 247]{a43} for a slighly different notion of a one-dimensional slowly oscillating function, and \cite{sarason0} for the notion of a slowly oscillating function $f : [0,\infty) \rightarrow {\mathbb C}$ at $0$ and $+\infty$):

\begin{defn}\label{drasko-norm}
Let $c\in {\mathbb C} \setminus \{0\},$ $\emptyset \neq I \subseteq {\mathbb R}^{n},$ ${\mathbb D} \subseteq I \subseteq {\mathbb R}^{n}$ and the set ${\mathbb D}$ be unbounded.
Define
\begin{align*}
A_{I}:=\bigl\{ {\bf \omega} \in {\mathbb R}^{n} \setminus \{0\} : {\bf \omega}+I \subseteq I\bigr\}.
\end{align*}
Then we say that a continuous
function $F:I \times X\rightarrow Y$ is $({\mathbb D},{\mathcal B})$-slowly oscillating
if and only if for each $B\in {\mathcal B}$ and $\omega \in A_{I},$ we have
\begin{align}\label{loza}
\lim_{|{\bf t}|\rightarrow +\infty,{\bf t}\in {\mathbb D}}\bigl\| F({\bf t}+\omega ; x)-F({\bf t} ; x)\bigr\|_{Y}=0,\quad \mbox{ uniformly in }x\in B.
\end{align}
\end{defn}

In other words, a continuous
function $F:I \times X\rightarrow Y$ is $({\mathbb D},{\mathcal B})$-slowly oscillating if and only if $F(\cdot;\cdot)$ is $(S,{\mathbb D},{\mathcal B})$-asymptotically 
$(\omega,1)$-periodic for all $\omega \in A_{I}.$ Clearly, we have $kA_{I}\subseteq A_{I}$ for all $k\in {\mathbb N}.$ 

If $X\in {\mathcal B},$ then we omit the term ``${\mathcal B}$'' from the notation and, if ${\mathbb D}=I,$ then we omit the term ``${\mathbb D}$'' from the notation; for example, if ${\mathbb D}=I$ and $F:I \rightarrow Y$ is $({\mathbb D},{\mathcal B})$-slowly oscillating with $X=\{0\}$, then we simply say that the function $F(\cdot)$ is slowly oscillating.

Now we would like to note that it is not so logical to study the class of $({\mathbb D},{\mathcal B},c)$-slowly oscillating functions
by replacing the term $\| F({\bf t}+\omega ; x)-F({\bf t} ; x)\|_{Y}$ in \eqref{loza} by the term $\| F({\bf t}+\omega ; x)-cF({\bf t} ; x)\|_{Y},$ where $c\in {\mathbb C} \setminus \{0\}.$ In actual fact, we have the following result which is clearly applicable if ${\mathbb D}=I=[0,\infty)^{n}$ or ${\mathbb D}=I={\mathbb R}^{n}$:

\begin{prop}\label{koq}
Let $c\in {\mathbb C} \setminus \{0\},$ $\emptyset \neq I \subseteq {\mathbb R}^{n},$ ${\mathbb D} \subseteq I \subseteq {\mathbb R}^{n}$ and the set ${\mathbb D}$ be unbounded.
Suppose that $A_{I}\subseteq 2A_{I}$ and $\omega'+{\mathbb D} \subseteq {\mathbb D}$ for all $\omega'\in A_{I}/2.$ 
Then the following holds:
\begin{itemize}
\item[(i)] If a continuous function $F : I \times X \rightarrow Y$ is $({\mathbb D},{\mathcal B},c)$-slowly oscillating, then $F\in C_{0,{\mathbb D},{\mathcal B}}(I \times X :Y).$
\item[(ii)] If, in addition to the above, we have $\omega +{\mathbb D} \subseteq {\mathbb D}$ for all $\omega \in A_{I},$ then  a continuous function $F : I \times X \rightarrow Y$ is $({\mathbb D},{\mathcal B},c)$-slowly oscillating if and only if $F\in C_{0,{\mathbb D},{\mathcal B}}(I \times X :Y).$
\end{itemize}
\end{prop}

\begin{proof}
To prove (i), suppose that $\omega'\in A_{I}$ and $B\in {\mathcal B};$ then there exists $\omega \in A_{I}$ such that $\omega'=2\omega.$ We have (${\bf t}\in I;$ $x\in B$):
\begin{align*}
F\bigl({\bf t} &+\omega '; x\bigr)-c^{2}F({\bf t};x)=F\bigl({\bf t} +2\omega ; x\bigr)-c^{2}F({\bf t}; x)
\\&=\bigl[ F\bigl({\bf t}+2\omega ; x\bigr)- c F\bigl({\bf t}+\omega ;x\bigr)\bigr]+c\bigl[ F\bigl({\bf t}+\omega ; x\bigr)- c F\bigl({\bf t};x\bigr)\bigr].
\end{align*}
The prescribed assumption $(A_{I}/2)+{\mathbb D}\subseteq {\mathbb D}$ implies ${\bf t}+\omega \in {\mathbb D},$ ${\bf t}\in {\mathbb D}$ and
$$
\lim_{|{\bf t}|\rightarrow +\infty,{\bf t}\in {\mathbb D}}\bigl\| F\bigl({\bf t}+\omega' ; x\bigr)-c^{2}F({\bf t} ; x)\bigr\|_{Y}=0,\quad \mbox{ uniformly in }x\in B.
$$
Subtracting the terms in the above limit equality and the limit equality \eqref{loza}, with the number $\omega$ replaced therein with the number $\omega'$, we get
$$
\lim_{|{\bf t}|\rightarrow +\infty,{\bf t}\in {\mathbb D}}\bigl\| \bigl(c^{2}-c\bigr) \cdot F({\bf t} ; x)\bigr\|_{Y}=0,\quad \mbox{ uniformly in }x\in B.
$$
This immediately implies (i) since $c\neq 1.$ To prove (ii), it suffices to apply (i) and observe that the assumption $\omega +{\mathbb D} \subseteq {\mathbb D}$ for all $\omega \in A_{I}$ implies
$$
\lim_{|{\bf t}|\rightarrow +\infty,{\bf t}\in {\mathbb D}}\bigl\| F({\bf t}+\omega ; x)\bigr\|_{Y}=0,\quad \mbox{ uniformly in }x\in B.
$$
\end{proof} 

Concerning the notion of a $({\mathbb D},{\mathcal B})$-slowly oscillating function, we would like to note that we do not require any kind of boundedness of function $F(\cdot;\cdot)$ here. In the classical approach, developed by D. Sarason \cite{sarason1} for the functions of form $f: [0,\infty) \rightarrow {\mathbb C},$ the boundedness of function $f(\cdot)$ is required a priori, which is not a direct consequence of definition since the function $f(t):=t^{\alpha},$ $ t\geq 0$ satisfies \eqref{loza} if $\alpha \in (0,1);$ the boundedness is obtained by applying the function $e^{i\cdot}$ after that (in other words, the function $t\mapsto e^{it^{\alpha}},$ $t\geq 0$ is slowly oscillating in the sense of \cite{sarason1}, for any $\alpha \in (0,1)$).
It is also worth noting that the global boundedness of function $f(\cdot)$ has not been used in the proof of \cite[Proposition 1]{sarason1}, as well as that the argumentation contained in the proof of this theorem can serve one to deduce the following result in the multi-dimensional setting; we will include all details of proof for the sake of completeness:

\begin{prop}\label{whiskey}
Suppose that $\emptyset \neq I \subseteq {\mathbb R}^{n}$ is an unbounded, closed set and the function $F: I\rightarrow Y$ is slowly oscillating.
Suppose, further, that the following condition holds:
\begin{itemize}
\item[(C)] For every $r>0$ and $\delta>0$, for every points ${\bf t},\ {\bf t}'\in I\setminus I_{r}$ with $|{\bf t}-{\bf t}'|<\delta$, and for every point $z\in (A_{I}+{\bf t}- {\bf t}') \cup  (A_{I}+{\bf t}'- {\bf t}) ,$ there exists $\eta_{z}>0$ such that $B(z,\eta_{z})\subseteq A_{I}.$ Here, $I_{r}\equiv \{ {\bf t}\in I :  |{\bf t}|\leq r\}$ ($r>0$).
\end{itemize}
Then the function $F(\cdot)$ is uniformly continuous.
\end{prop}

\begin{proof}
Suppose that the function $F(\cdot)$ is not uniformly continuous. Since the set $I$ is closed, the set $I\cap B(0,r)$ is compact for all positive real numbers $r>0;$ hence, 
the following holds:
\begin{itemize}
\item[(D)] 
There exists a positive real number $\epsilon>0$ such that, for every positive real numbers
$\delta>0$ and $r>0,$ there exist ${\bf t},\ {\bf t}'\in I  \setminus I_{r}$ such that $|{\bf t}-{\bf t}'|<\delta$ and $\| F({\bf t})-F({\bf t}')\|_{Y}>\epsilon.$
\end{itemize}
Using conditions (C) and (D), as well as the fact that the function $F(\cdot)$ is slowly oscillating, we can inductively construct the sequences $(\omega_{k})$ in $ A_{I},$
$({\bf t}_{k})$
in $I$ 
and $(\eta_{k})$ in $(0,\infty)$
such that
$\lim_{k\rightarrow +\infty}\eta_{k}=0,$ $\lim_{k\rightarrow \infty}|{\bf t}_{k}|=+\infty,$
$B(\omega_{k},\eta_{k})\subseteq B(\omega_{k+1},\eta_{k+1})\subseteq A_{I},$ $k\in {\mathbb N}$ and
$\| F({\bf t}_{k})-F({\bf t}_{k}+{\bf t})\|_{Y}\geq \epsilon/2,$ provided $k\in {\mathbb N}$
and ${\bf t}\in B(\omega_{k},\eta_{k});$ it is only worth noting here that, in each step of this construction, we can choose the point
$\omega_{k+1}$ to be $\omega_{k}+({\bf t_{k}}-{\bf t}_{k}')$ or $\omega_{k}+({\bf t_{k}}'-{\bf t}_{k})$, where the points ${\bf t}_{k}$ and ${\bf t}_{k}'$ are already chosen points from $I$ with sufficiently large absolute values, satisfying additionally that $\| F({\bf t}_{k})-F({\bf t}_{k}')\|_{Y}>\epsilon$ and $|{\bf t_{k}}'-{\bf t}_{k}|\leq 1/k.$ 
Due to the Cantor theorem, there exists a unique number ${\bf t}'\in \bigcap_{k\in {\mathbb N}} B(\omega_{k},\eta_{k}).$
This implies $\| F({\bf t}_{k})-F({\bf t}_{k}+{\bf t}')\|_{Y}\geq \epsilon/2$ for all $k\in {\mathbb N},$
which is a contradiction since the function $F(\cdot)$ is slowly oscillating and ${\bf t}'\in A_{I}.$
\end{proof}

Using this result, the interested reader may simply transfer the statement of \cite[Proposition 2]{sarason1} to the higher dimensions, as well; details can be left to the interested readers. For more details about the life and professional work of D. Sarason, we refer the reader to the communication paper \cite{dsar} by S. R. Garcia.

We continue by observing that, 
in the infinite-dimensional setting, there exists a bounded, uniformly continuous, slowly oscillating function $F: [0,\infty) \rightarrow Y$ whose range is not relatively compact in $Y$:

\begin{example}\label{pier} (\cite{pierro,brazil})
Let $Y:=c_{0}$ be the space of all numerical sequences tending to zero, equipped with the sup-norm. Set
$$
F(t):=\Biggl(\frac{4n^{2}t^{2}}{(t^{2}+n^{2})^{2}} \Biggr)_{n\in {\mathbb N}},\ t\geq 0.
$$
Then the function $F(\cdot)$ is bounded, uniformly continuous and satisfies 
\begin{align*}
\|F(t+\tau)-F(t)\|\leq t^{-4}+4\frac{\tau^{2}}{t^{2}},\ t>0,\ \tau\geq 0,
\end{align*}
so that the function $F(\cdot)$ is slowly oscillating. 
We already know that the range of $F(\cdot)$ is not 
relatively compact in $Y.$
\end{example}

The following notion is also meaningful:

\begin{defn}\label{drasko-arx}
Let 
${\mathbb D}_{j} \subseteq I \subseteq {\mathbb R}^{n}$ and the set ${\mathbb D}_{j}$ be unbounded ($1\leq j\leq n$).
Define
\begin{align*}
B_{I}:=\bigl\{ (\omega_{1},...,  \omega_{n}) \in ({\mathbb R} \setminus \{0\})^{n} :  \omega_{j}e_{j}+I \subseteq I \mbox{ for all }j\in {\mathbb N}_{n}\bigr\}.
\end{align*}
Then we say that a continuous
function $F:I \times X\rightarrow Y$ is
$({\mathcal B},({\mathbb D}_{j})_{j\in {\mathbb N}_{n}})$-slowly oscillating if and only if for each $(\omega_{1},...,\omega_{n})\in B_{I}$ and $ j\in {\mathbb N}_{n}$ we have
\begin{align*}
\lim_{|{\bf t}|\rightarrow +\infty,{\bf t}\in {\mathbb D}_{j}}\bigl\|F({\bf t}+{\bf \omega}_{j}e_{j};x)-F({\bf t};x)\bigr\|_{Y}=0,\quad \mbox{ uniformly in }x\in B.
\end{align*}
\end{defn} 

In other words, a continuous
function $F:I \times X\rightarrow Y$ is $({\mathcal B},({\mathbb D}_{j})_{j\in {\mathbb N}_{n}})$-slowly oscillating if and only if $F(\cdot;\cdot)$ is $(S,{\mathcal B})$-asymptotically $({\bf \omega}_{j},1,{\mathbb D}_{j})_{j\in {\mathbb N}_{n}}$-periodic for all tuples $({\bf \omega}_{1},..., {\bf \omega}_{n})\in B_{I}.$ Clearly, we have $kB_{I}\subseteq B_{I}$ for all $k\in {\mathbb N}.$

In \cite{multi-arx}, we have  investigated the following topics in connection with $(S,{\mathbb D},{\mathcal B})$-asymptotically 
$(\omega,c)$-periodic functions and $(S,{\mathcal B})$-asymptotically $({\bf \omega}_{j},c_{j},{\mathbb D}_{j})_{j\in {\mathbb N}_{n}}$-periodic functions:
\begin{itemize}
\item[(i)] the invariance under the operation of uniform convergence,
\item[(ii)] the convolution invariance,
\item[(iii)] the invariance under reflections at zero,
\item[(iv)] the translation invariance,
\item[(v)] the pointwise products with the scalar-valued functions of the same type, etc.
\end{itemize}
All these statements can be simply reformulated for the notion introduced in Definition \ref{drasko-norm} and Definition \ref{drasko-arx} (with $c=1$; $c_{j}=1$ for all $j\in {\mathbb N}_{n}$). We will 
skip all applications based on the use of results concerning the above-mentioned topics, like those established to d'Alembert formula and the heat equation in ${\mathbb R}^{n}$; see \cite{multi-arx} for more details.

\section{On the relations between quasi-asymptotical $c$-almost periodicity and remote $c$-almost periodicity}\label{qaap}

In this section, we will first remind the readers of the notion of quasi-asymptotical $c$-almost
periodicity (we do not use the assumption $I'\subseteq I$ here; cf. also \cite[Definition 3.3]{BIMV} for the one-dimensional setting):

\begin{defn}\label{nakaza} (\cite{multi-arx})
Suppose that 
${\mathbb D} \subseteq I \subseteq {\mathbb R}^{n}$, $\emptyset  \neq I \subseteq {\mathbb R}^{n},$ 
$\emptyset  \neq I'\subseteq {\mathbb R}^{n},$ 
the sets ${\mathbb D}$ and $I'$ are unbounded,
$F : I \times X \rightarrow Y$ is a continuous function and $I +I' \subseteq I.$ Then we say that:
\begin{itemize}
\item[(i)]\index{function!${\mathbb D}$-quasi-asymptotically Bohr $({\mathcal B},I',c)$-almost periodic}
$F(\cdot;\cdot)$ is ${\mathbb D}$-quasi-asymptotically $({\mathcal B},I',c)$-almost periodic if and only if for every $B\in {\mathcal B}$ and $\epsilon>0$
there exists $l>0$ such that for each ${\bf t}_{0} \in I'$ there exists ${\bf \tau} \in B({\bf t}_{0},l) \cap I'$ such that there exists a finite real number $M(\epsilon,\tau)>0$ such that
\begin{align}\label{emojmarko145m}
\bigl\|F({\bf t}+{\bf \tau};x)-cF({\bf t};x)\bigr\|_{Y} \leq \epsilon,\mbox{ provided } {\bf t},\ {\bf t}+\tau \in {\mathbb D}_{M(\epsilon,\tau)},\ x\in B.
\end{align}
\item[(ii)] \index{function!${\mathbb D}$-quasi-asymptotically $({\mathcal B},I',c)$-uniformly recurrent}
$F(\cdot;\cdot)$ is ${\mathbb D}$-quasi-asymptotically $({\mathcal B},I',c)$-uniformly recurrent if and only if for every $B\in {\mathcal B}$
there exist a sequence $({\bf \tau}_{k})$ in $I'$ and a sequence $(M_{k})$ in $(0,\infty)$ such that $\lim_{k\rightarrow +\infty} |{\bf \tau}_{k}|=\lim_{k\rightarrow +\infty}M_{k}=+\infty$ and
\begin{align}\label{tastapub}
\lim_{k\rightarrow +\infty}\sup_{{\bf t},{\bf t}+{\bf \tau}_{k}\in {\mathbb D}_{M_{k}};x\in B} \bigl\|F({\bf t}+{\bf \tau}_{k};x)-c F({\bf t};x)\bigr\|_{Y} =0.
\end{align}
\end{itemize}

If $I'=I,$ then we also say that
$F(\cdot;\cdot)$ is ${\mathbb D}$-quasi-asymptotically $({\mathcal B},c)$-almost periodic (${\mathbb D}$-quasi-asymptotically $({\mathcal B},c)$-uniformly recurrent); furthermore, if $X\in {\mathcal B},$ then it is also said that $F(\cdot;\cdot)$ is ${\mathbb D}$-quasi-asymptotically $(I',c)$-almost periodic (${\mathbb D}$-quasi-asymptotically $(I',c)$-uniformly recurrent). If $I'=I$ and $X\in {\mathcal B}$, then we also say that $F(\cdot;\cdot)$ is ${\mathbb D}$-quasi-asymptotically  $c$-almost periodic (${\mathbb D}$-quasi-asymptotically $c$-uniformly recurrent). We remove the prefix ``${\mathbb D}$-'' in the case that ${\mathbb D}=I$, remove the prefix ``$({\mathcal B},)$''  in the case that $X\in {\mathcal B}$ and remove the prefix ``$c$-'' if $c=1.$ We will accept these terminological agreements for the notion introduced in Definition \ref{ris} below, as well.
\end{defn}

Now we would like to take  a closer look at the equations \eqref{emojmarko145m} and \eqref{tastapub}. We first observe that it is completely irrelevant whether we will write that there exists a finite real number $M(\epsilon,\tau)>0$ such that \eqref{emojmarko145m} holds, or more concisely,
\begin{align}\label{emojmarko145m-rem}
\limsup_{|{\bf t}|\rightarrow +\infty,{\bf t}\in {\mathbb D}}\sup_{x\in B}\bigl\|F({\bf t}+{\bf \tau};x)-cF({\bf t};x)\bigr\|_{Y}\leq \epsilon,
\end{align}
i.e.,
\begin{align*}
\lim_{s\rightarrow +\infty}\sup_{|{\bf t}|\geq s,{\bf t}\in {\mathbb D}; x\in B}\bigl\|F({\bf t}+{\bf \tau};x)-cF({\bf t};x)\bigr\|_{Y} \leq \epsilon.
\end{align*}
It is also very simple to show that it is completely irrelevant whether we will write that there exists a finite real number $M(\epsilon,\tau)>0$ such that \eqref{tastapub} holds, or more concisely,
$$
\lim_{k\rightarrow +\infty}\limsup_{|{\bf t}|\rightarrow +\infty,{\bf t}\in {\mathbb D}}
\sup_{x\in B}\bigl\|F({\bf t}+{\bf \tau}_{k};x)-c F({\bf t};x)\bigr\|_{Y} =0,
$$
i.e.,
$$
\lim_{k\rightarrow +\infty}\lim_{s\rightarrow +\infty}\sup_{|{\bf t}|\geq s, {\bf t}\in {\mathbb D}; x\in B}\bigl\|F({\bf t}+{\bf \tau}_{k};x)-c F({\bf t};x)\bigr\|_{Y} =0.
$$

The special case $c=1,$ $X\in {\mathcal B}$ and ${\mathbb D}=I=I'={\mathbb R}^{n}$ has been considered in \cite{BIMV}, \cite{brazil} and \cite{multi-arx}, where a ${\mathbb D}$-quasi-asymptotically Bohr $({\mathcal B},I',c)$-almost periodic function is simply called quasi-asymptotically almost periodic. In this case, the above consideration shows that the notion of quasi-asymptotical almost periodicity is equivalent with the notion of remote almost periodicity considered by F. Yang and C. Zhang in \cite[Definition 1.1; (1) and (3)]{f-yang}; see also the pioneering paper \cite{sarason}, where D. Sarason has analyzed the complex-valued remotely almost periodic functions defined on the real line, and the paper \cite{c-zhang}, where C. Zhang and L. Jiang
have analyzed the class of remotely almost periodic sequences (see also \cite{rabin}).

As our former analyses show (see also the research article \cite{xie} by R. Xie and C. Zhang), a quasi-asymptotically almost periodic function $F : {\mathbb R}^{n} \rightarrow {\mathbb C}$ need not be uniformly continuous, so that the notion introduced in 
\cite[Definition 1.1; (2)]{f-yang} is not satisfactory to a certain extent (see also S. Zhang, D. Piao \cite[Definition 2.1]{s-zhang} and the first sentence after \cite[Definition 1.1]{f-yang}, where the authors have assumed a priori that a remotely almost periodic function $F:  {\mathbb R}^{n} \rightarrow X$ is uniformly continuous).
It is our strong belief that it is much better to analyze both: the general classes of ${\mathbb D}$-quasi-asymptotically Bohr $({\mathcal B},I',c)$-almost periodic type functions which are not uniformly continuous on ${\mathcal B}$ and the corresponding classes of 
${\mathbb D}$-quasi-asymptotically Bohr $({\mathcal B},I',c)$-almost periodic type functions which are uniformly continuous on ${\mathcal B}$ (in a certain sense):

\begin{defn}\label{ris}
Suppose that $F : I \times X \rightarrow Y$ is a continuous function.
\begin{itemize}
\item[(i)] It is said that $F(\cdot;\cdot)$ is ${\mathbb D}$-remotely $({\mathcal B},I',c)$-almost periodic if and only if $F(\cdot;\cdot)$ is ${\mathbb D}$-quasi-asymptotically Bohr $({\mathcal B},I',c)$-almost periodic and for each $B\in {\mathcal B}$ the function $F(\cdot;\cdot)$ is uniformly continuous on $I\times B;$ that is
\begin{align*}
(\forall B\in {\mathcal B})& \ (\forall \epsilon>0) \ (\exists \delta > 0) \ (\forall {\bf t'},\ {\bf t''} \in I)\ \bigl(\forall x',\ x''\in B\bigr)
\\& \Biggl( \bigl| {\bf t'}-{\bf t''} \bigr|+\bigl \| x-x' \bigr\|<\delta \Rightarrow \Bigl\| F\bigl({\bf t'} ; x'\bigr)-F\bigl({\bf t''} ; x''\bigr)\Bigr\|_{Y}<\epsilon  \Biggr).
\end{align*}
\item[(ii)] It is said that $F(\cdot;\cdot)$ is ${\mathbb D}$-remotely $({\mathcal B},I',c)$-uniformly recurrent if and only if $F(\cdot;\cdot)$ is ${\mathbb D}$-quasi-asymptotically Bohr $({\mathcal B},I',c)$-uniformly recurrent and for each $B\in {\mathcal B}$ the function $F(\cdot;\cdot)$ is uniformly continuous on $I\times B.$
\item[(iii)] It is said that $F(\cdot;\cdot)$ is ${\mathbb D}$-remotely $({\mathcal B},I',c)$-almost periodic of type $1$ if and only if $F(\cdot;\cdot)$ is ${\mathbb D}$-quasi-asymptotically Bohr $({\mathcal B},I',c)$-almost periodic and
\begin{align}
\notag (\forall B\in {\mathcal B})& \ (\forall \epsilon>0) \ (\exists \delta > 0) \ (\forall {\bf t'},\ {\bf t''} \in I)\ (\forall x\in B)
\\\label{sqwa}& \Biggl( \bigl| {\bf t'}-{\bf t''} \bigr|<\delta \Rightarrow \Bigl\| F\bigl({\bf t'} ; x\bigr)-F\bigl({\bf t''} ; x\bigr)\Bigr\|_{Y}<\epsilon  \Biggr).
\end{align}
\item[(iv)] It is said that $F(\cdot;\cdot)$ is ${\mathbb D}$-remotely $({\mathcal B},I',c)$-uniformly recurrent of type $1$ if and only if $F(\cdot;\cdot)$ is ${\mathbb D}$-quasi-asymptotically Bohr $({\mathcal B},I',c)$-uniformly recurrent and \eqref{sqwa} holds.
\end{itemize}
\end{defn}

It is clear that any ${\mathbb D}$-remotely $({\mathcal B},I',c)$-almost periodic (${\mathbb D}$-remotely $({\mathcal B},I',c)$-uniformly recurrent) function is ${\mathbb D}$-remotely $({\mathcal B},I',c)$-almost periodic of type $1$ (${\mathbb D}$-remotely $({\mathcal B},I',c)$-uniformly recurrent) of type $1$. The converse statement holds provided that the function $F(\cdot;\cdot)$ is Lipshitzian with respect to the second argument:

\begin{prop}\label{zivotlajf}
Suppose that $F : I \times X \rightarrow Y$ is a continuous function and for each set $B\in {\mathcal B}$ there exists a finite real constant $L_{B}>0$ such that 
\begin{align}\label{lips-lips}
\bigl\| F\bigl({\bf t};x'\bigr)-F\bigl({\bf t};x''\bigr)\bigr\|_{Y} \leq L_{B}\bigl\| x'-x''\bigr\|,\quad {\bf t}\in I,\ x',\ x'' \in  B.
\end{align}
If $F(\cdot;\cdot)$ is ${\mathbb D}$-remotely $({\mathcal B},I',c)$-almost periodic of type $1$ (${\mathbb D}$-remotely $({\mathcal B},I',c)$-uniformly recurrent of type $1$), then $F(\cdot;\cdot)$ is ${\mathbb D}$-remotely $({\mathcal B},I',c)$-almost periodic (${\mathbb D}$-remotely $({\mathcal B},I',c)$-uniformly recurrent).
\end{prop}

\begin{proof}
Let the set $B\in {\mathcal B}$ be given and let $L_{B}>0$ satisfy \eqref{lips-lips}.
The proof is a simple consequence of the corresponding definitions and the following decomposition (${\bf t},\ {\bf t}\in I;$ $x',\ x''\in B$):
\begin{align*}
\Bigl\| F\bigl({\bf t'} ; x'\bigr)-F\bigl({\bf t''} ; x''\bigr)\Bigr\|_{Y}&\leq \Bigl\| F\bigl({\bf t'} ; x'\bigr)-F\bigl({\bf t'} ; x''\bigr)\Bigr\|_{Y}+\Bigl\| F\bigl({\bf t'} ; x''\bigr)-F\bigl({\bf t''} ; x''\bigr)\Bigr\|_{Y}
\\& \leq L_{B}\bigl\| x'-x''\bigr\|+\Bigl\| F\bigl({\bf t'} ; x''\bigr)-F\bigl({\bf t''} ; x''\bigr)\Bigr\|_{Y}.
\end{align*}
\end{proof}

Further on, we want to notice that we do not require any type of boundedness of function $F(\cdot)$ in
Definition \ref{nakaza} and Definition \ref{ris}; for example, an application of the Lagrange mean value theorem yields that for each fixed real number $\sigma \in (0,1)$ we have 
$
| (t+\tau)^{\sigma}-t^{\sigma} | \leq \tau \sigma t^{\sigma-1},$ $t>0,$ $\sigma\geq 0,
$
so that the function $t\mapsto t^{\sigma},$ $t\geq 0$ is remotely almost periodic in the sense of Definition \ref{ris}, as we have already discussed for slowly oscillating functions. 
In connection with the unboundedness of function $F(\cdot)$ in these definitions, we would like to present the following illustrative example:

\begin{example}\label{merak}
Let us recall that
A. Haraux and P. Souplet have proved, in \cite[Theorem 1.1]{haraux}, that the function $f: {\mathbb R}\rightarrow {\mathbb R},$ given by
\begin{align}\label{gader}
f(t):=\sum_{n=1}^{\infty}\frac{1}{n}\sin^{2}\Bigl(\frac{t}{2^{n}} \Bigr)\, dt,\quad t\in {\mathbb R},
\end{align}
is uniformly continuous, uniformly recurrent and unbounded; futhermore, this function is even, Weyl $p$-almost automorphic for any finite exponent $p\geq 1$ and satisfies that for each number $\tau\in {\mathbb R}$ the function $f(\cdot +\tau)-f(\cdot)$ belongs to the space $ANP({\mathbb R}: {\mathbb C})$
consisting of all almost periodic complex-valued functions whose Bohr's spectrum is contained in the set ${\mathbb R}\setminus \{0\};$ see \cite{densities} for the notion and more details. 

We will prove here that the function $f(\cdot)$ is not quasi-asymptotically almost periodic in the sense of Definition \ref{nakaza} and \cite[Definition 3.3]{BIMV}. If we assume the contrary, then for each positive real number $\epsilon>0$ there exists a finite real number $l>0$ such that
any subinterval $I' \subseteq {\mathbb R}$ of length $l$ contains a number $\tau \in I'$ such that there exists a finite real number $M(\epsilon,\tau)>0$ so that $|f(t +\tau)-f(t)|\leq \epsilon$ for $|t|\geq M(\epsilon,\tau).$ Fix such a number $\tau;$ then the function $t\mapsto f(t +\tau)-f(t),$ $t\in {\mathbb R}$ is
almost periodic so that an application of the supremum formula for almost periodic functions \cite[Theorem 2.1.1(xi)]{nova-mono} yields that
$$
\sup_{t\in {\mathbb R}}|f(t+\tau)-f(t)|=\sup_{t\geq M(\epsilon,\tau)}|f(t +\tau)-f(t)|\leq \epsilon.
$$
This implies that the number $\tau$ is the usual $\epsilon$-period of function $f(\cdot)$ so that the function $f(\cdot)$ is almost periodic by definition, which is a contradiction.
\end{example}

If we denote by $Q-AAP_{buc}({\mathbb R}^{n} : Y)$ the space consisting of all bounded, uniformly continuous quasi-asymptotically almost periodic functions $F : {\mathbb R}^{n} \rightarrow Y,$ then we know from the foregoing that $Q-AAP_{buc}({\mathbb R}^{n} : Y)$ coincides with the space of all
uniformly continuous (usually, we assume this as a blank hypothesis)
remotely almost periodic functions ${\mathcal R}{\mathcal A}{\mathcal P}({\mathbb R}^{n} : Y).$ We know therefore that
$Q-AAP_{buc}({\mathbb R}^{n} : {\mathbb C})$
is exactly 
the closed subalgebra of  
$C_{b}({\mathbb R}^{n}: {\mathbb C})$
generated by the space of all almost periodic functions $F : {\mathbb R}^{n} \rightarrow {\mathbb C}$ and the space of all slowly oscillating functions $F : {\mathbb R}^{n} \rightarrow {\mathbb C};$
this means that, for every $\epsilon>0$ and for every $F\in Q-AAP_{buc}({\mathbb R}^{n} : {\mathbb C})$, we can always find two almost periodic functions $G_{i} : {\mathbb R}^{n} \rightarrow {\mathbb C}$ ($i=1,2$) and two slowly oscillating functions  
$Q_{i} : {\mathbb R}^{n} \rightarrow {\mathbb C}$ ($i=1,2$) such that $\| F -[G_{1}+Q_{1}+G_{2}Q_{2}]\|_{\infty}<\epsilon$ (\cite{sarason, f-yang}).
The proof of this important result is based on the use of certain results from the theory of $C^{\ast}$-algebras concerning the Gelfand spaces of multiplicative linear functionals of ${\mathcal R}{\mathcal A}{\mathcal P}({\mathbb R}^{n} : {\mathbb C});$
it could be very enticing to 
extend this result for the functions defined on the general regions $I\subseteq {\mathbb R}^{n}.$

The results obtained in \cite[Proposition 2.1, Proposition 2.2]{c-zhang1} provide 
new characterizations of bounded, uniformly continuous quasi-asymptotically almost periodic functions $F : {\mathbb R}^{n} \rightarrow Y,$ while
\cite[Proposition 2.3]{c-zhang1} and \cite[Proposition 2.3]{s-zhang} provide 
new characterizations of bounded, uniformly continuous quasi-asymptotically almost periodic functions $F : {\mathbb R} \rightarrow Y.$ 
On the other hand, the results obtained in \cite[Theorem 3.1, Theorem 3.2, Proposition 3.4]{BIMV}, the composition principles obtained in \cite[Theorem 3.3, Theorem 3.4]{BIMV} and the result obtained in \cite[Proposition 2.15]{brazil}
provide 
new characterizations of remotely ($c$-)almost periodic functions $F : I \rightarrow Y,$ $I\subseteq {\mathbb R}$ (it is worth noting that \cite[Proposition 3.4(ii)]{BIMV} can be used to substantially shorten the proof of \cite[Lemma 3.6]{s-zhang}),
while the results obtained in \cite[Proposition 3.2, Proposition 3.5, Theorem 3.6]{multi-arx} provide 
new characterizations of remotely ($c$-)almost periodic functions $F : I \rightarrow Y,$ $I\subseteq {\mathbb R}^{n}$ (and certain two-parameter analogues). For example, using \cite[Theorem 3.1(ii)]{BIMV} with $c=1$ and our analysis contained in the final paragraph of \cite[Section 3]{multi-arx}, we immediately get
that
$$
AA\bigl({\mathbb R}^{n} :Y \bigr) \cap  {\mathcal R}{\mathcal A}{\mathcal P}\bigl({\mathbb R}^{n} : Y\bigr)=AP\bigl( {\mathbb R}^{n} :Y \bigr),
$$
where $AP({\mathbb R}^{n}:Y)$ and $AA({\mathbb R}^{n}:Y)$ denote the space of all almost periodic functions from ${\mathbb R}^{n}$ into $Y$ and the space of all almost automorphic functions from ${\mathbb R}^{n}$ into $Y$, respectively; see \cite{marko-manuel-aa} for the notion.

From application point of view, it is incredibly important to emphasize that  \cite[Proposition 3.4]{BIMV} can be used to profile some statements 
concerning the invariance of remote $c$-almost periodicity under the
actions of convolution products, since the uniform continuity is preserved under the actions of convolution products in the equations \cite[(3.1); (3.2)]{BIMV}; these results seem to new and not considered elsewhere even for the usual remote almost periodicity ($c=1$). This enables one to provide numerous important applications in the study of time-remotely almost periodic solutions 
for various classes of the abstract (degenerate) Volterra integro-differential equations (see also \cite[Section 4]{brazil}, where we have analyzed quasi-asymptotically almost periodic solutions of the abstract nonautonomous differential equations of first order; the question whether the obtained solutions are uniformly continuous is not so simple to be answered and requires further analyses).
We will provide only one illustrative application of this type:

\begin{example}\label{mp}
Let $\Omega$ be a bounded domain in ${\mathbb R}^{n},$ $b>0,$ $m(x)\geq 0$ a.e. $x\in \Omega$, $m\in L^{\infty}(\Omega),$ $1<p<\infty$ and $X:=L^{p}(\Omega).$
Suppose that the operator $A:=\Delta -b $ acts on $X$ with the Dirichlet boundary conditions, and
$B$ is the multiplication operator by the function $m(x).$
Then the multivalued linear operator
${\mathcal A}:=AB^{-1}$ satisfies condition \cite[(P)]{nova-mono} with $\beta=1/p$ and some finite constants $c,\ M>0.$ Therefore, we are in a position to analyze the existence and uniqueness of remotely almost periodic solutions 
of the following Poisson heat equation on the real line
\[\left\{
\begin{array}{l}
\frac{\partial}{\partial t}[m(x)v(t,x)]=(\Delta -b )v(t,x) +f(t,x),\quad t\in {\mathbb R},\ x\in {\Omega},\\
v(t,x)=0,\quad (t,x)\in [0,\infty) \times \partial \Omega ,\\
\end{array}
\right.
\]
and the following Poisson heat equation on the non-negative real line
\[\left\{
\begin{array}{l}
\frac{\partial}{\partial t}[m(x)v(t,x)]=(\Delta -b )v(t,x) +f(t,x),\quad t\geq 0,\ x\in {\Omega};\\
v(t,x)=0,\quad (t,x)\in [0,\infty) \times \partial \Omega ,\\
 m(x)v(0,x)=u_{0}(x),\quad x\in {\Omega}
\end{array}
\right.
\]
in the space  $X,$ by using the substitution $u(t,x)=m(x)v(t,x)$ and passing to the corresponding linear differential inclusions of first order; see \cite{nova-mono} for more details. We can also analyze the existence and uniqueness of remotely almost periodic solutions for certain classes of fractional Poisson type equations and the semilinear analogues for all above-mentioned classes of abstract degenerate equations (see \cite[Theorem 3.3, Theorem 3.4]{BIMV}).
\end{example}

For the sequel, we need to recall the following result from \cite{multi-arx}:

\begin{lem}\label{gasdas}
\begin{itemize}
\item[(i)] Let ${\bf \omega}\in I \setminus \{0\},$ $c\in {\mathbb C} \setminus \{0\},$ $|c|\leq 1,$
${\bf \omega}+I \subseteq I$ and ${\mathbb D} \subseteq I \subseteq {\mathbb R}^{n}.$ Set $I':=\omega \cdot {\mathbb N}.$
If a continuous
function $F:I \times X\rightarrow Y$ is $(S,{\mathbb D},{\mathcal B})$-asymptotically 
$(\omega,c)$-periodic, then the function $F(\cdot; \cdot)$ is ${\mathbb D}$-quasi-asymptotically $({\mathcal B},I',c)$-almost periodic.
\item[(ii)] Let ${\bf \omega}_{j}\in {\mathbb R} \setminus \{0\},$ $c_{j}\in {\mathbb C} \setminus \{0\},$
${\bf \omega}_{j}e_{j}+I \subseteq I,$ 
${\mathbb D}_{j} \subseteq I \subseteq {\mathbb R}^{n},$ the set ${\mathbb D}_{j}$ be unbounded ($1\leq j\leq n$) and the set ${\mathbb D}$ consisting of all tuples ${\bf t}\in {\mathbb D}_{n}$ such that ${\bf t}+\sum_{i=j+1}^{n}\omega_{i}e_{i}\in {\mathbb D}_{j}$ for all $j\in {\mathbb N}_{n-1}$
be unbounded in ${\mathbb R}^{n}$. Set $\omega:=\sum_{j=1}^{n}\omega_{j}e_{j},$ 
$I':=\omega \cdot {\mathbb N}$
and $c:=\prod_{j=1}^{n}c_{j}.$
If $F:I \times X \rightarrow Y$ is $(S,{\mathcal B})$-asymptotically $({\bf \omega}_{j},c_{j},{\mathbb D}_{j})_{j\in {\mathbb N}_{n}}$-periodic, $|c|\leq 1$ and $\omega \in I,$ then  
the function $F(\cdot;\cdot)$ is ${\mathbb D}$-quasi-asymptotically $({\mathcal B},I',c)$-almost periodic.
\end{itemize}
\end{lem}

Let us observe that, if a continuous function $F : I \times X \rightarrow Y$ is
${\mathbb D}$-quasi-asymptotically $({\mathcal B},I'_{i},c)$-almost periodic for $i=1,2,$ then the function $F(\cdot;\cdot)$ is ${\mathbb D}$-quasi-asymptotically $({\mathcal B},I'_{1}\cup I_{2}',c)$-almost periodic (a similar statement holds for ${\mathbb D}$-quasi-asymptotical $({\mathcal B},I',c)$-uniform recurrence). Keeping this in mind, the subsequent result follows immediately from Lemma \ref{gasdas}:

\begin{prop}\label{gasdas1}
\begin{itemize}
\item[(i)] Let 
${\mathbb D} \subseteq I \subseteq {\mathbb R}^{n}$ and the set ${\mathbb D}$ be unbounded.
If a continuous
function $F:I \times X\rightarrow Y$ is $({\mathbb D},{\mathcal B})$-slowly oscillating, then the function $F(\cdot; \cdot)$ is ${\mathbb D}$-quasi-asymptotically $({\mathcal B},I')$-almost periodic with
$$
I':=\bigl\{\omega \cdot {\mathbb N} \, ; \, \omega \in A_{I} \bigr\}.
$$
\item[(ii)] Let 
${\mathbb D}_{j} \subseteq I \subseteq {\mathbb R}^{n},$ the set ${\mathbb D}_{j}$ be unbounded ($1\leq j\leq n$) and for each tuple ${\bf \omega}= (\omega_{1},...,\omega_{n})\in B_{I}$ the set ${\mathbb D}_{{\bf \omega}}$ consisting of all tuples ${\bf t}\in {\mathbb D}_{n}$ such that ${\bf t}+\sum_{i=j+1}^{n}\omega_{i}e_{i}\in {\mathbb D}_{j}$ for all $j\in {\mathbb N}_{n-1}$
be unbounded in ${\mathbb R}^{n}$. Suppose that the set ${\mathcal D}\equiv \bigcap_{\omega \in B_{I}}{\mathbb D}_{{\bf \omega}}$  is unbounded,
$$
I':=\bigl\{\omega \cdot {\mathbb N} \, ; \, \omega \in B_{I} \cap I \bigr\},
$$
and $c:=\prod_{j=1}^{n}c_{j}.$
If $F:I \times X \rightarrow Y$ is $({\mathcal B},({\mathbb D}_{j})_{j\in {\mathbb N}_{n}})$-slowly oscillating, then  
the function $F(\cdot;\cdot)$ is ${\mathbb D}$-quasi-asymptotically $({\mathcal B},I')$-almost periodic.
\end{itemize}
\end{prop}

It is clear that every slowly oscillating function $F : I \rightarrow Y$, where $I$ is $[0,\infty)^{n}$ or ${\mathbb R}^{n},$ is quasi-asymptotically almost periodic, which immediately follows from Proposition \ref{gasdas1}.

We will not consider here
the differentiation and integration of multi-dimensional remotely $c$-almost periodic functions (see \cite[Subsection 2.4]{marko-manuel-ap} for the related results concerning multi-dimensional $({\mathrm R},{\mathcal B})$-almost periodic type functions,
and 
\cite[Proposition 2.3]{s-zhang} for a result concerning the first anti-derivatives of one-dimensional remotely almost periodic functions).
Concerning the existence of mean value,
the 
boundedness of a remotely $c$-almost periodic function $F(\cdot)$ is almost inevitable to be assumed in order to ensure the existence of finite mean value of $F(\cdot)$. We feel it is our duty to emphasize that the 
proof of \cite[Proposition 2.4]{s-zhang}, a statement which considers the existence and properties of mean value of one-dimensional remotely almost periodic functions defined on the whole real line, is not correct and contains several important mistakes: 
\begin{itemize}
\item[1.] The estimate directly after the equation \cite[(2.12)]{s-zhang} is not correct since the term ``$2G(l+s_{0})$'' has to be written here as ``$2G(2l+s_{0}+a)$'', which causes several serious and unpleasant consequences for the remainder of the proof.  
\item[2.] It is not clear the meaning of the number $T_{0}$ in the equations \cite[(2.13)-(2.14)]{s-zhang}.
\item[3.] The existence of mean value, stated in the equation \cite[(2.15)]{s-zhang}, is given without any reasonable explanation; see also the proof of \cite[Theorem 1.3.1, pp. 32-34]{188}, where the correct proof of the existence of mean value is given for the usually considered class of almost periodic functions (besides these observations, we would like to note that the uniform continuity of function $f(\cdot)$
has not been used in the proof of the above-mentioned proposition). 
\end{itemize}
Keeping in mind these observations, it follows that the problem of existence or non-existence of mean value of remotely almost periodic functions is still unsolved; we want also to
emphasize that our structural results established in the first two sections of \cite{maulen} and the part of the third section of \cite{maulen} before subsection 3.1 of this paper, where it has directly been assumed that a remotely almost periodic function has a mean value, remain completely true by assuming additionally that 
any considered remotely almost periodic function has a mean value.

Now we will prove the existence of a bounded, uniformly continuous slowly oscillating function $f : [0,\infty) \rightarrow c_{0}$ which does  not have mean value, which clearly marks that
the calculations given in \cite[Proposition 2.4]{s-zhang} are not true:

\begin{example}\label{vojko}
Define $f : [0,\infty) \rightarrow c_{0}$ by $f(t):=(e^{-t/n})_{n\in {\mathbb N}},$ $t\geq 0.$ In \cite[Example 2.2]{brindle}, D. Brindle has proved that the function $f(\cdot)$ is bounded, uniformly continuous and slowly oscillating (albeit we have found some minor typographical errors in this example, the obtained conclusions are correct; we can use the inequality $1-e^{-x}\leq x,$ $x\geq 0$ here). If we assume that the limit 
$$
\lim_{t\rightarrow +\infty}\frac{1}{t}\int^{t}_{0}f(s)\, ds
$$ 
exists in $c_{0},$
then it can be simply approved that this limit has to be equal to the  zero sequence, so that we would have
\begin{align}\label{21}
\lim_{t\rightarrow +\infty}\sup_{n\in {\mathbb N}}\Biggl[ \frac{n}{t}\Bigl(1-e^{-(t/n)} \Bigr)\Biggr]=0.
\end{align}
If we assume that $t\geq 1$ and $t\in [n,n+1)$ for some integer $n\in {\mathbb N}$, then we have 
\begin{align*}
\frac{n}{t}\Bigl(1-e^{-(t/n)} \Bigr) \geq \frac{n}{n+1}\Bigl(1-e^{-(n/n)} \Bigr)\geq \frac{1}{2}\bigl(1-e^{-1} \bigr)
\end{align*}
and therefore 
\begin{align*}
\sup_{n\in {\mathbb N}}\Biggl[ \frac{n}{t}\Bigl(1-e^{-(t/n)} \Bigr)\Biggr]\geq \frac{1}{2}\bigl(1-e^{-1} \bigr),\quad t\geq 1,
\end{align*}
which clearly contradicts \eqref{21}.
\end{example}

The interested reader may try to construct an example of a bounded, uniformly continuous slowly oscillating function $f : [0,\infty) \rightarrow {\mathbb C}$ without mean value (it is our strong belief that such a function really exists; see also \cite[Section 2.2]{brindle}).

At the end of this section, we would like to point out that the proofs of \cite[Theorem 2.36]{marko-manuel-ap} and \cite[Theorem 2.28]{multi-ce} concerning the extensions of multi-dimensional ($c$-)almost periodic type functions do not work for quasi-asymptotically $(I',c)$-almost periodic functions and remotely $(I',c)$-almost periodic functions. Without going into full details, let us only note that the situation is much simpler for slowly oscillating functions, when we can construct many different extensions of a slowly oscillating function $F : I \rightarrow Y$ to the whole Euclidean space ${\mathbb R}^{n};$ for example, if a slowly oscillating function $f : [0,\infty) \rightarrow Y$ is given in advance, we can extend it linearly to the interval $[-r,0],$ where $r>0$ is an arbitrary real number, and after that we can extend the obtained function by zero outside the interval $[-r,\infty).$

\section{Applications to the integro-differential equations}\label{apply}

This section is devoted to some applications of our results to the abstract Volterra integro-differential equations and the ordinary differential equations.
 \vspace{0.1cm}

1. In \cite[Theorem 3.4]{c-zhang1}, C. Zhang and L. Jiang have analyzed remotely almost periodic solutions of the perturbed heat equation
\begin{align}
\notag u_{t}&=\sum_{i=1}^{m+n}\Bigl[u_{x_{i}x_{i}}+b_{i}(x,t)u_{x_{i}} \Bigr]-c(x,t)u=f(x,t),\quad (x,t) \in {\mathbb R}^{n+m}_{T};
\\\label{role} & u(x,0)=\varphi(x),\quad x\in {\mathbb R}^{n+m},
\end{align}
following the method proposed by A. Friedman \cite{friedman}; see also the boundary value problem considered in \cite[Lemma 3.3]{f-yang}, which can be also reconsidered in our context. Since \cite[Lemma 3.1]{c-zhang1} (see also the proof of \cite[Proposition 2.41]{c-zhang1}) and \cite[Corollary 3.2, Lemma 3.3]{c-zhang1} can be reformulated for multi-dimensional remotely $c$-almost periodic functions, the argumentation contained in the proof of \cite[Theorem 3.4]{c-zhang1} shows that the following holds (we define the spaces ${\mathcal RAP}_{c}({\mathbb R}^{n} \times \overline{{\mathbb R}^{m}_{T}})$ and $ {\mathcal RAP}_{c}({\mathbb R}^{n+m})$ similarly as in \cite{c-zhang1}, with the use of difference $\cdot -c\cdot$ in place of difference $\cdot-\cdot$): 

\begin{thm}\label{kopkop}
If the functions 
$f(x,t),$ $b_{i}(x,t),\ \partial b_{i}/\partial x_{j}(x,t)$ ($j=1,....,n+m$) and $c(x,t)$ belong to the space ${\mathcal RAP}_{c}({\mathbb R}^{n} \times \overline{{\mathbb R}^{m}_{T}})$ and the functions $\varphi,\ \partial \varphi/\partial x_{j}$ belong to the space $ {\mathcal RAP}_{c}({\mathbb R}^{n+m}),$ then there exists a unique solution $u(x,t)$ of \eqref{role} which can be written as a finite sum of functions belonging to the space ${\mathcal RAP}_{c}({\mathbb R}^{n} \times \overline{{\mathbb R}^{m}_{T}}).$ 
\end{thm}

Let us also point out that the statement of \cite[Proposition 2.2]{c-zhang1} does not
hold for remotely $c$-almost periodic functions unless $c=1.$ We will not analyze the inverse parabolic problems here.

2. The convolution invariance of multi-dimensional quasi-asymptotically $c$-almost periodic functions have recently been analyzed in \cite[Theorem 3.6]{multi-arx}. For our purposes, the following special case of this theorem will be sufficiently enough:

\begin{lem}\label{milenko1}
Suppose that ${\mathbb D}={\mathbb R}^{n}$, $h\in L^{1}({\mathbb R}^{n}),$ $\emptyset  \neq I'\subseteq {\mathbb R}^{n}$ is unbounded and $F : {\mathbb R}^{n} \rightarrow Y$ is a bounded continuous function.  
Then the function $(h\ast F)(\cdot),$ given by 
\begin{align*}
(h\ast F)({\bf t}):=\int_{{\mathbb R}^{n}}h(\sigma) F({\bf t}-\sigma)\, d\sigma,\quad {\bf t}\in {\mathbb R}^{n},
\end{align*}
is well defined and bounded; furthermore, if $F(\cdot;\cdot)$ is 
quasi-asymptotically $(I',c)$-almost periodic, resp. quasi-asymptotically $(I',c)$-uniformly recurrent,
then the function $(h\ast F)(\cdot)$ is likewise quasi-asymptotically $(I',c)$-almost periodic,
resp. quasi-asymptotically $(I',c)$-uniformly recurrent.
\end{lem}

We want also to note that the uniform continuity of $F(\cdot)$ implies the uniform continuity of $(h\ast F)(\cdot),$ as can be simply shown; Lemma \ref{milenko1} and this observation will be crucial for the following application, which has recently been considered in \cite{marko-manuel-aa} for the multi-dimensional almost automorphic type functions. We will consider only integrated semigroups here (\cite{a43, knjiga}).

Suppose that
$k\in\mathbb{N}$, $a_{\alpha}\in\mathbb{C}$, $0\!\leq\! |\alpha|\!\leq\! k$, $a_{\alpha}\neq 0$
for some $\alpha$ with $|\alpha|=k$, $P(x)=\sum_{|\alpha|\leq k}a_{\alpha}i^{|\alpha|}x^{\alpha}$,
$x\in\mathbb{R}^n$, $P(\cdot)$ is an elliptic polynomial, i.e., there exist $C>0$ and $L>0$
such that $|P(x)|\geq C|x|^k$, $|x|\geq L$, $\omega:=\sup_{x\in\mathbb{R}^n}\Re(P(x))<\infty$, and
$X$ is $C_b(\mathbb{R}^n)$ or
$BUC(\mathbb{R}^n)$ [the space of bounded uniformly continuous functions $f : {\mathbb R}^{n} \rightarrow {\mathbb C}$ equipped with the sup-norm]. Define
$$
P(D):=\sum_{|\alpha|\leq k}a_{\alpha}f^{(\alpha)}
\text{ and } Dom(P(D)):=\bigl\{f\in E:P(D)f\in E\text{ distributionally}\bigr\}
$$
and assume that $n_X>n/2$.
Then we know that the operator $P(D)$ generates an exponentially bounded $r$-times integrated semigroup $(S_{r}(t))_{t\geq 0}$
in $X$ for any $r>n_X.$ Furthermore, we know that for each $t\geq 0$ there exists a function $f_{t} \in L^{1}({\mathbb R}^{n})$ such that
$$
\bigl[S_{r}(t)f\bigr](x):=\bigl(f_{t}\ast f\bigr)(x),\quad x\in {\mathbb R}^{n},\ f\in X.
$$
Fix a number $t_{0} \geq 0$ and assume that
the function $X\ni f$ is 
remotely $c$-almost periodic, for example. By the foregoing, we get that the function $x\mapsto [S_{r}(t_{0})f](x),$ $x\in {\mathbb R}^{n}$ is likewise remotely $c$-almost periodic, which means that there exists a unique $X$-valued continuous function $t\mapsto u(t),$ $t\geq 0$ such that
$\int^{t}_{0}u(s)\, ds \in Dom(P(D))$ for every $t\geq 0$ and
$$
u(t)=P(D)\int^{t}_{0}u(s)\, ds -\frac{t^{r}}{\Gamma (r+1)}f,\quad t\geq 0;
$$ 
furthermore, the solution $t\mapsto u(t),$ $t\geq 0$ has the property that its trajectory consists solely of remotely $c$-almost periodic functions. Suppose now that
$r\in {\mathbb N},$ $f\in Dom(P(D)^{r})$ and the function
$$
\bigl(f,\ P(D)f,\cdot \cdot \cdot, P(D)^{r}f\bigr)
$$
is remotely $c$-almost periodic. Then the function
\begin{align}\label{kojo-kojo}
u(t):=S_{r}(t)P(D)^{r}f+\frac{t^{r-1}}{(r-1)!}P(D)^{r-1}f+\cdot \cdot \cdot +t P(D)f+f,\quad t\geq 0
\end{align}
is a unique continuous $X$-valued function which satisfies that 
$\int^{t}_{0}u(s)\, ds \in Dom(P(D))$ for every $t\geq 0$ and
$$
u(t)=P(D)\int^{t}_{0}u(s)\, ds -f,\quad t\geq 0;
$$ 
due to \eqref{kojo-kojo} and our assumptions, for every fixed number $t\geq 0$ we have that the function $ u(t)\in X$ is remotely $c$-almost periodic.

\subsection{An application in mathematical biology}\label{biology}

The application is closely related with our recent investigation of remotely almost periodic solutions of ordinary differential equations \cite{maulen} (some results established in this research article, like \cite[Theorem 3]{maulen}, can be reformulated for remotely $c$-almost periodic functions but we will skip all related details for simplicity). 

Let us recall that the Chapman-Richards functions and the Chapman-Richards
equations are incredibly important in the mathematical
biology. The Chapman-Richards functions generalize
monomolecular
functions and Gompertz functions, while the Chapman-Richards equations generalize the logistic equations. It is well known that the
Chapman-Richards equations have many applications in the
forestry, thanks to their flexibility and important analytical
features.

Of concern is the following Richard-Chapman equation with an external perturbation $f(\cdot):$
\begin{align}\label{(3.12)}
x^{\prime}(t)= x (t)\Bigl[ a (t) - b (t) x^{\theta} (t)\Bigr] + f (t),
\end{align}
where $\theta \geq 0.$ We will analyze slowly oscillating functions in the classical sense here, so that any such a function will be bounded and uniformly continuous in this part. 
Consider the following hypotheses:
\begin{itemize}
\item[(H1)] $a (t),$  $b (t) $ and $f (t)$ are slowly oscillating functions;
\item[(H2)]$ 0 <\alpha\leq 	a(t)	\leq A,$ $0 <\beta\leq	b(t)	\leq B,$  $0 < f(t) < F ;$
\item[(H3)] With $\omega=A^{-1}[\beta-\gamma^{(1+\theta)/\theta}F]$ and $\gamma=B/\alpha,$ we have $(1+\theta)F\gamma^{1/\theta}\theta^{-1}\alpha^{-1}<1$ and $\beta (1+\theta)B\theta^{-1}<1$.
\end{itemize}			

We need the following auxiliary lemma, which generalizes \cite[Lemma 3.5]{f-yang} (see also \cite[Lemma 4]{maulen}):

\begin{lem}\label{tehno}
Suppose that $\alpha>0,$ the functions $a : {\mathbb R} \rightarrow [\alpha,\infty)$ and $f : {\mathbb R} \rightarrow {\mathbb R}$ are slowly oscillating. Then the function
$$
t\mapsto F(t)\equiv \int_{-\infty}^{t}e^{-\int^{t}_{s}a(r)\, dr}f(s)\, ds,\quad t\in {\mathbb R}
$$
is slowly oscillating, as well.
\end{lem}

\begin{proof}
Let $\omega \in {\mathbb R}\setminus \{0\}.$
The proof that the function $F(t)$ is slowly oscillating as $ t\rightarrow -\infty$ follows from the existence of  a sufficiently large number $t_{0}>0$ such that $|a(t+\omega)-a(t)|+|f(t+\omega)-f(t)|<\epsilon$ provided that $|t|>t_{0},$ as well as the following calculation:
\begin{align*}
&|F(t+\omega)-F(t)|
\\&=\Biggl| \int^{0}_{-\infty}e^{-\int^{t}_{s+t}a(r+\omega)\, dr}f(t+s+\omega)\, ds- \int^{0}_{-\infty}e^{-\int^{t}_{s+t}a(r)\, dr}f(t+s)\, ds\Biggr|
\\& \leq    \int^{0}_{-\infty}e^{-\int^{t}_{s+t}a(r+\omega)\, dr}|f(t+s+\omega)-f(t+s)|\, ds
\\& +\|f \|_{\infty} \int^{0}_{-\infty}\Biggl|  e^{-\int^{t}_{s+t}a(r+\omega)\, dr}-e^{-\int^{t}_{s+t}a(r)\, dr} \Biggr|\, ds
\\& \leq (\epsilon /\alpha)+ \|f \|_{\infty} \int^{0}_{-\infty}e^{\alpha s}\Biggl|  1-e^{\int^{t}_{s+t}[a(r+\omega)-a(r)]\, dr} \Biggr|\, ds
\\& \leq (\epsilon /\alpha)+ \|f \|_{\infty}   \int^{0}_{-\infty}e^{\alpha s}\Biggl| \int^{t}_{s+t}[a(r+\omega)-a(r)]\, dr \Biggr| e^{\Bigl| \int^{t}_{s+t}[a(r+\omega)-a(r)]\, dr \Bigr| }\, ds
\\& \leq (\epsilon /\alpha)+ \|f \|_{\infty}\epsilon \int^{0}_{-\infty}|s|e^{(\alpha +\epsilon)s}\, ds= (\epsilon /\alpha)+ \|f \|_{\infty}\epsilon  (\alpha +\epsilon)^{2},\quad t<-t_{0}.
\end{align*}
The proof that the function $F(t)$ is slowly oscillating as $ t\rightarrow +\infty$ is a little incorrect in \cite[Lemma 3.5]{f-yang} but we can apply a trick from \cite[Remark 2.2]{c-zhang-guo} here. Strictly speaking, we can use the same decomposition and calculation as above  but we need to divide first the interval of integration $(-\infty,0]$ into two subintervals $(-\infty,-M]$ and $[-M,0],$ where $M>0$ is a sufficiently large real number such that $\int^{0}_{-\infty}e^{\alpha s}\, ds <\epsilon/2.$ 
\end{proof}

Keeping in mind Lemma \ref{tehno}, the fact that the space of real-valued slowly oscillating functions is closed under pointwise products and sums, as well as the fact that for each positive slowly oscillating function $f : {\mathbb R} \rightarrow (0,\infty)$ and for each real number $r>0$ the function $f^{r}:  {\mathbb R} \rightarrow (0,\infty)$ is also slowly oscillating,
we can repeat verbatim the
argumentation contained in the proof of \cite[Theorem 6]{maulen} in order to see that the following result holds true:

\begin{thm}\label{bukva}
Suppose that the hypotheses \emph{(H1)}-\emph{(H3)} hold. Then the equation \eqref{(3.12)} has a unique slowly oscillating	 solution $\phi^{\ast}(t)$ satisfying $\gamma^{-1/\theta}\leq \phi^{\ast}(t) \leq \omega^{-1/\theta}$ for all $t\in {\mathbb R}.$
\end{thm}

\end{document}